\documentclass[noams]{compositio}
\usepackage[utf8]{inputenc}
\usepackage[english]{babel} % Om Nederlandse woord afbreking te krijgen
\usepackage{graphicx}     % Om figuren te laten zien
\usepackage{hyperref}     % Voor externe links. 
\usepackage{pst-node}
\usepackage{tikz-cd}
\usepackage{mathtools}
\usepackage{MnSymbol}
\usepackage{color,soul}
\usepackage{url}
\usepackage{ragged2e}

\newtheorem{theorem}{Theorem}[section]
\newtheorem{corollary}[theorem]{Corollary}
\newtheorem{lemma}[theorem]{Lemma}
\newtheorem{proposition}[theorem]{Proposition}
\newtheorem{example}[theorem]{Example}

\theoremstyle{definition}
\newtheorem{definition}[theorem]{Definition}
\newtheorem{remark}[theorem]{Remark}

\newcommand{\R}{\mathbb{R}}
\newcommand{\Z}{\mathbb{Z}}

\newcommand{\F}{\mathbb{F}}

\newcommand{\bx}{\mathbf{x}}
\newcommand{\by}{\mathbf{y}}

\newcommand{\V}{\mathcal{V}}

\newcommand{\Y}{\mathcal{Y}}
\newcommand{\CR}{\mathcal{C}}

\newcommand{\ad}{\mathfrak{a}}
\newcommand{\m}{\mathfrak{m}}

\newcommand{\lcm}{\operatorname{lcm}}
\newcommand{\ord}{\operatorname{ord}}

\newcommand{\coeff}{\operatorname{coeff}}
\newcommand{\Mon}{\operatorname{Mon}}

\newcommand{\Gal}{\operatorname{Gal}}
\newcommand{\Orb}{\operatorname{Orb}}
\newcommand{\Div}{\operatorname{Div}}

\newcommand{\Jac}{\operatorname{Jac}}
\newcommand{\im}{\operatorname{im}}

\newcommand{\e}{\operatorname{e}}

\newcommand{\vm}{\vec{m}}
\newcommand{\vd}{\vec{d}}

\newcommand{\ovl}[1]{\overline{#1}}
\newcommand{\wt}[1]{\widetilde{#1}}
\renewcommand{\contentsname}{Contents\\{\footnotesize\normalfont(A table
of contents should normally not be included)}}
\begin{document}

\title{The Krull-Remak-Schmidt decomposition of commutative group algebras}
\author{Robert Christian Subroto}
\email{bobby.subroto@ru.nl}
\address{Department of Digital Security\\Radboud University\\Nijmegen\\The Netherlands}
%
%\dedication{A dedication can be included here.}
\classification{20C05, 20C20, 13M05, 13M10. 
%At least one subject code is required. Please refer to \url{http://www.ams.org/msc/} for a list of codes.
}
\keywords{group algebras, modular representation theory, commutative algebra, polynomials rings}
%\thanks{This file documents \pkg{compositio} version \Fileversion\ and was last revised \Filedate.}

\begin{abstract}
We provide the Krull-Remak-Schmidt decomposition of group algebras of the form $k[G]$ where $k$ is a field, which includes fields with prime characteristic, and $G$ a finite abelian group.
We achieved this by studying the geometric equivalence of $k[G]$ which we call circulant coordinate rings.
\end{abstract}

\maketitle

%\vspace*{6pt}\tableofcontents  % for this guide only.
% A table of contents should normally not be included

\section{Introduction}
\label{sec:introduction}

The object of interest of this paper is the group algebra $k[G]$, where $k$ is a field and $G$ a finite abelian group.
We view this algebra as a module over itself, which is both Noetherian and Artinian by finiteness of $G$.
As such, $k[G]$ satisfies the criteria for the Krull-Remak-Schmidt theorem, which states that such a module admits a decomposition consisting of indecomposable modules, and that such a decomposition is unique up to isomorphism.

In this paper, we present the complete Krull-Remak-Schmidt decomposition of the group algebra $k[G]$.
We achieved this by studying a geometric object called circulant coordinate rings, which is the geometric equivalence of $k[G]$.

\subsection{Outline}

In Section \ref{SECTION:Prerequisites_and_scope}, we cover some basic results regarding the structure of multivariate polynomial rings.
We also cover the Krull-Remak-Schmidt theorem and some related statements.

In Section \ref{SECTION:Circulant_Rings}, we introduce circulant rings, which is the geometric equivalence of the group algebra $k[G]$.
We also cover some cases at both ends of the spectrum of the Krull-Remak-Schmidt decomposition of circulant rings, which are the indecomposable and the semisimple circulant rings.

In Section \ref{SECTION:Krull-Remak-Schmidt}, we cover the Krull-Remak-Schmidt decomposition of circulant rings in its full generality, where we use results from modular representation theory.

Ine Section \ref{SECTION:Application}, we apply the Krull-Remak-Schmidt decomposition to the special case when we are working over a finite field.
We also give a concrete example.

\subsection{Notation}

The cardinality or size of a set $S$ is denoted as $\# S$.
The set of all positive integers strictly greater than $0$ is denoted as $\Z_{>0}$.
Similarly we define $\Z_{\geq 0}$.
If $S$ is a finite index set, we refer to $R^{\oplus S}$ as the direct sum of copies of the ring $R$ indexed by $S$.

\paragraph{\textbf{Rings and field extensions}}
A ring $R$ is called a unitary ring if it contains a unit element.
For such a unitary ring $R$, we denote the multiplicative group of invertible elements of $R$ by $R^*$.
All rings in this paper are assumed to be unitary rings.

For a field $k$, we denote its algebraic closure by $\ovl{k}$. 
Given a field extension $l / k$, we denote the degree of the extension as $[l : k]$.
Moreover, given a tuple $\bx := ( x_1, \ldots ,x_n ) \in \ovl{k}^n$, the smallest field extension over $k$ containing all $x_i$ is denoted by $k(\bx)$.
We denote $\mu_m$ as the set of all $m$-th roots of unity in $\ovl{k}$.
When $l/k$ is a Galois extension, we  write the Galois group as $\Gal(l/k)$.

\paragraph{\textbf{Multivariate polynomials and vanishing set}}
We denote $k[X_1,...,X_n]$ by $A^n_k$.
The set of monomials is referred to as $\Mon(n)$.
Given a polynomial $f \in A^n_k$, we denote the coefficient of $f$ at the monomial $M \in \Mon(n)$ as $\coeff_M(f)$.   
For a monomial $M := \prod_{i=1}^n  X_i^{m_i} \in \Mon(n)$, the total degree of $M$ is defined as $\deg(M) := \sum_{i=1}^n m_i$. 
The $i$-th partial degree of $M$ is defined as $\deg_i(M) := m_i$.
We extend the notions of partial- and total degrees to general polynomials.
for a polynomial $f \in A^n_k$, we define the total degree of $f$ as: 
\begin{align*}
\deg(f) := \max(\deg(M) : M \in \Mon(\vm) \text{ and } \coeff_M(f) \neq 0).
\end{align*}
Similarly, the $i$-th partial degree of $f$ is defined as: 
\begin{align*}
\deg_i(f) := \max(\deg_i(M) : M \in \Mon(\vm) \text{ and } \coeff_M(f) \neq 0).
\end{align*}

\begin{example}
Take $f(X_1, X_2, X_3) := X_1^2 - 6 X_1 X_2^2 X_3^3 - 4 X_1 X_2^4 - 7 X_1^2 \in A^3_{\R}$.
Then $\deg(f) = 6$, $\deg_1(f) = 2$, $\deg_2(f) = 4$ and $\deg_3(f) = 3$.
\end{example}

For an ideal $\ad$ of $A^n_k$, we define the vanishing set of $\ad$ as
\begin{align*}
\V(\ad) := \left\{ \bx \in \ovl{k}^n \mid f(\bx) = 0 \text{ for all } f \in \ad \right\}.
\end{align*}

\paragraph{\textbf{Number theory}}
For an integer $m > 0$, we denote the ring of integers modulo $m$ as $\Z_m$.
This is not to be confused with the $p$-adic ring of integers which also uses the same notation in the case when $m = p$.
When referring to an element $x \in \Z_m$, we refer to the integer representation of the residue class of $x$, hence $x \in \{ 0, 1, \ldots, m-1 \}$.

Let $p$ be a prime number and $m$ be any positive integer.
The $p$-adic valuation of $m$ is defined as $v_p(m) := \max\{ j \in \Z_{\geq 0} : p^j \mid m \}$.
We define $r_p(m) = \frac{m}{v_p(m)}$, which is the greatest divisor of $m$ coprime to $p$.

\paragraph{\textbf{Tuple notation}}

Let $\vm := (m_1, \ldots, m_n) \in \Z_{\geq 0}^n$ be a non-zero $n$-tuple.
We define the product ring $\Z_{\vm} := \Z_{m_1} \times \ldots \times \Z_{m_n}$.
For a prime number $p$, we define $v_p(\vm) := (v_p(m_1), \ldots, v_p(m_n))$ and $r_p(\vm) := (r_p(m_1), \ldots, r_p(m_n))$.
Also, we define $p^{v_p(\vm)} := (p^{v_p(m_1)}, \ldots, p^{v_p(m_n)})$.
For two $n$-tuples $\vm = (m_1, \ldots, m_n)$ and $\vm' = (m'_1, \ldots, m'_n)$, we say that $\vm > \vm'$ if $m_i \geq m'_i$ for all $1 \leq i \leq n$, and where there exists at least one $1 \leq j \leq n$ such that $m_j > m'_j$.
We say that $\vm$ is \emph{absolutely larger} than $\vm'$ if $m_i > m'_i$ for all $1 \leq i \leq n$, which we write as $\vm \gg \vm'$.
Similarly, we define $\vm < \vm'$ and $\vm \ll \vm'$ (\emph{absolutely smaller}).

\section{Prerequisites}
\label{SECTION:Prerequisites_and_scope}

In this section, we discuss some results related to the structure of polynomial rings which will prove useful in this paper.
Moreover, we briefly cover the Krull-Remak-Schmidt theorem.

\subsection{Structure of multivariate polynomial rings}

\begin{definition}
\label{DEF:polyrep}
Given $n$-tuples $j := (j_1,...,j_n), t := (t_1,...,t_n) \in \Z_{\geq 0}^n$, we define the polynomial $\Y^j_t := \prod_{i=1}^n (X_i^{t_i} - 1)^{j_i}$ in $A^n_k := k[X_1, \ldots, X_n]$.

Moreover, we define 
\begin{align*}
\mathcal{R}_t := \{ f \in A^n_k \mid \deg_i(f) < t_i \text{ for all } 1 \leq i \leq n  \}.
\end{align*}

Also, we define
\begin{align*}
\mathcal{R}_t \Y_t^j := \{ f \cdot \Y^j_t \mid f \in \mathcal{R}_t \},
\end{align*}
which we view as an additive subgroup of $A^n_k$.
\end{definition}

\begin{theorem}
\label{THM:polyStructureTheorem}
Let $t$ be an $n$-tuple.
We have the following statements:
\begin{itemize}
\item[1.]	For an $n$-tuple $j$, if $f \in \mathcal{R}_t \Y^j_t$ is non-zero, then $t_i \cdot j_i \leq \deg_i(f) < t_i \cdot (j_i + 1)$ for all $1 \leq i \leq n$;
\item[2.] 	Given two $n$-tuples $j,j'$ such that $j \neq j'$, we have $\mathcal{R}_t \Y^j_t \cap \mathcal{R}_t \Y^{j'}_t = \{ 0 \}$;
\item[3.]	Let $\e_i$ be the $i$-th standard unit vector in $\Z_{\geq 0}^n$, then $\mathcal{R}_{t (\e_i + 1)} = \mathcal{R}_t \oplus \mathcal{R}_t \Y_t^{\e_i}$ for all $1 \leq i \leq n$;
\item[4.] 	Let $j$ and $t$ be $n$-tuples such that all their entries are non-zero, then $\mathcal{R}_{t \cdot j} := \bigoplus_{0^n \leq y \ll j} \mathcal{R}_t \Y^y_t$.
\end{itemize}
\end{theorem}

\begin{proof}
We prove these statements separately.

\paragraph{Statement 1}
Observe that $\deg_i(\Y_t^j) = j_i \cdot t_i$ for all $1 \leq i \leq n$. 
Also for any polynomial $f \in A^n_k$, we have that $\deg_i(f \cdot \Y_t^j) = \deg_i(f) + \deg_i(\Y_t^j)$.
The statement follows from the fact that $0 \leq \deg_i(f) < t_i$ for all $f \in \mathcal{R}_t$.

\paragraph{Statement 2} 
Without loss of generality, assume that $j_1 \neq j'_1$.
Statement 1 implies that the $1$-th partial degree of the non-zero polynomials in $\mathcal{R}_t \Y^j_t$ and $\mathcal{R}_t \Y^{j'}_t$ cannot be in the same range.
This proves the claim.

\paragraph{Statement 3}
Without loss of generality, take $i = 1$.
Consider the monomial $M := \prod_{l=1}^n X_l^{m_l}$ where $0 \leq m_l < t_l$ for all $2 \leq l \leq n$, and $t_1 \leq m_1 < 2 t_1$.
The goal is to show that such a monomial is contained in $\mathcal{R}_t \oplus \mathcal{R}_t \Y_t^{\e_1}$.

Observe that $\Y_t^{\e_1} = X_1^{t_1} - 1$.
Since $M' :=  X_1^{m_1 - t_1} \cdot \prod_{l=2}^n X_l^{m_l} \in \mathcal{R}_t$, we have that $M' \cdot \Y_t^{\e_1} \in \mathcal{R}_t \Y_t^{\e_1}$.
Observe that
\begin{align*}
M' \cdot \Y_t^{\e_1}
= X_1^{m_1 - t_1} \cdot \prod_{i=2}^n X_i^{m_i} \cdot \Y_t^{\e_1} 
= X_1^{m_1 - t_1} \cdot \prod_{i=2}^n X_i^{m_i} \cdot ( X_1^{t_1} - 1)
= M - M'.
\end{align*}
This implies that $M$ is contained in $\mathcal{R}_t \oplus \mathcal{R}_t \Y_t^{\e_1}$.
Since this is true for any such monomial $M$, the claim follows.

\paragraph{Statement 4}
Observe that every tuple $j = (j_1, \ldots, j_n)$ with non-zero entries equals $(1, \ldots, 1) + \sum_{i=1}^n (j_i - 1) \e_i$.
Hence the claim follows from inductively applying Statement 3, given that the claim holds for the case $j = (1, \ldots, 1)$.

Observe that for $j = (1, \ldots, 1)$, the only tuple which is absolutely smaller than $j$ is $j^0 := (0, \ldots, 0)$.
Note that
\begin{align*}
\mathcal{R}_t \Y_t^{j^0} 
= \left\{ f \cdot \prod_{i=1}^n (X_i^{t_i} - 1)^0 \mid f \in \mathcal{R}_t \right\}
= \{ f \mid f \in \mathcal{R}_t \} = \mathcal{R}_t,
\end{align*}
which proves the claim.
\end{proof}

As a consequence of Theorem \ref{THM:polyStructureTheorem}, for every element $f \in \mathcal{R}_{t \cdot j}$, for each $0^n \leq y \ll j$, there exists  a unique $f_{y} \in \mathcal{R}_t \Y_t^y$ such that $f = \sum_{0^n \leq y \ll j} f_{y}$.

\begin{definition}
For every $f \in \mathcal{R}_{t \cdot j}$, for each $0^n \ll y \ll j$, we define $f_y$ as the \textbf{$\Y_t^y$-component} of $f$.
\end{definition}

\subsection{Krull-Remak-Schmidt theorem and semilocal rings}

In this subsection, we briefly cover the Krull-Remak-Schmidt theorem.
We start by covering some basic notions from module theory.

Consider a unitary ring $R$ (not necessarily commutative), and let $M$ be an $R$-module.
We say that $M$ is \textbf{indecomposable} if it cannot be written as the direct sum of two proper $R$-submodules.
A straightforward example of indecomposable modules are \textbf{simple} modules, which are modules with the $0$-module as the only proper submodule.

Indecomposable $R$-modules are important, as they form the building blocks for a large family of $R$-modules.
This is shown in the Krull-Remak-Schmidt theorem.

\begin{theorem}[\textbf{Krull-Remak-Schmidt}]
Let $M$ be a Noetherian and Artinian $R$-module over some unitary ring $R$.
Then the following statements are true:
\begin{itemize}
\item[1.]	There exists an $R$-module decomposition $M \cong \bigoplus_{i = 1}^s M_i$ of $M$ where the $M_i$ are indecomposable $R$-modules;
\item[2.]	The decomposition is unique up to isomorphism, i.e., given another decomposition $\bigoplus_{j = 1}^{s'} M'_j$ of $M$ consisting of indecomposable $R$-modules, it satisfies the following properties:
\begin{itemize}
\item[i.]	$s = s'$;
\item[ii.]	there is a permutation $\sigma$ of $\{1, \ldots, s \}$ such that $M'_i \cong M_{\sigma(i)}$ for any $1 \leq i \leq s$.
\end{itemize}
\end{itemize}
\end{theorem}

\begin{proof}
See \cite[Theorem~1.4.7.]{schneider2012modular}.
\end{proof}

One can view a ring $R$ as a module over itself, where the $R$-scaling is the ring multiplication.
In this setting, left (or right) submodules of $R$ are the left (or right) ideals of $R$.

\begin{theorem}
\label{THM:groupalg_Noether}
The group algeba $k[G]$ is Artinian and Noetherian if and only if $G$ is finite.
\end{theorem}

\begin{proof}
See \cite{connell1963group}.
\end{proof}

As a result of Theorem \ref{THM:groupalg_Noether}, a group algebra $k[G]$ where $G$ is a finite abelian group satisfies the conditions of the Krull-Remak-Schmidt theorem.
As such, $k[G]$ can be  decomposed into indecomposable modules which are unique up to isomorphism.
The main scope of this paper is to provide such a decomposition of $k[G]$.

\begin{lemma}
\label{LEM:local_indecomposable}
Let $R$ be a local ring, then $R$ is indecomposable as a module over itself.
\end{lemma}

\begin{proof}
This is a direct consequence of \cite[Proposition~3.1]{jacobson2009basic}.
\end{proof}

The following results will prove useful in finding the decomposition.

\begin{theorem}
\label{THM:important_result}
Let $(R_i)_{i \in I}$ be a family of local rings where the index set $I$ is finite, and consider the product ring $R :=\bigoplus_{i \in I} R_i$.
Let us write $\ovl{R_j}$ as the subset $\{ (r_i)_{i \in I} \in R \mid r_i = 0 \text{ for all } i \neq j \}$.
Then the $\ovl{R_i}$ are indecomposable $R$-submodules of $R$, and $R$ admits the Krull-Remak-Schmidt decomposition
$R = \bigoplus_{i \in I} \ovl{R_i}$.
\end{theorem}

\begin{proof}
Clearly $\ovl{R_i}$ is closed under addition and multiplication.
Also for any $r \in R$ and $v \in \ovl{R_i}$, one can see by definition of $\ovl{R_i}$ that $r \cdot v \in \ovl{R_i}$.
Hence $\ovl{R_i}$ is an ideal of $R$, and thus admits an $R$-module structure.

We show that $\ovl{R_i}$ is indecomposable.
Let $\varepsilon_i : R \to R_i$ be the projection map, which induces the additive group isomorphism $\varepsilon_i' := \varepsilon_i \mid_{\ovl{R_i}} \colon \ovl{R_i} \to R_i$.
Observe that we then immediately have $\varepsilon_i \times \varepsilon'_i \colon R \times \ovl{R_i} \to R_i \times R_i$.
Hence we can equivalently show that $R_i$ is indecomposable as module over itself.
This is true since $R_i$ is local (Lemma \ref{LEM:local_indecomposable}).
\end{proof}

\begin{corollary}
\label{COR:important_result}
Using the same notation as in Theorem \ref{THM:important_result}, let $R'$ be a ring, and assume we have a ring isomorphism $\varphi \colon R \to R'$.
Then $\varphi(\ovl{R_i})$ are indecomposable $R'$-submodules of $R'$, and $R'$ admits the Krull-Remak-Schmidt decomposition
$R = \bigoplus_{i \in I} \varphi(\ovl{R_i})$.
\end{corollary}

\section{Circulant rings and group algebras}
\label{SECTION:Circulant_Rings}

We introduce circulant rings, which are the geometric equivalence of group algebras.
This is the same notion as the ones defined in \cite{subroto2024wedderburndecompositioncommutativesemisimple}.
We also cover indecomposable and semisimple group algebras.

\subsection{Geometric approach to group algebras}

\begin{definition}
Let $k$ be a field, and consider the $n$-tuple $\vm := (m_1, \ldots ,m_n) \in \Z_{>0}^n$.
The \textbf{cyclotomic ideal over $k$ with parameters $\vm$} is defined as the ideal $( X_1^{m_1} - 1, \ldots , X_n^{m_n} - 1 )$ in $A^n_k$, and is denoted $\ad_{\vm}$.

The \textbf{circulant ring} over $k$ with paramaters $\vm$ is defined as the coordinate ring $A^n_{k} / \ad_{\vm}$, which we denote by $\CR_{\vm / k}$.
\end{definition}

\begin{proposition}
\label{PROP:representatives}
The set $\mathcal{R}_{\vm / k}$ as defined in Definition \ref{DEF:polyrep} is a set of representatives of $\CR_{\vm / k}$, which we call the \textbf{standard set of representatives}.
\end{proposition}

\begin{proof}
Apply long division on the partial polynomial degrees.
\end{proof}

\begin{proposition}
Let $\vm := (m_1, \ldots ,m_n) \in \Z_{>0}^n$.
Then 
\begin{align*}
\V(\ad_{\vm}) = \mu_{r_p(m_1)} \times \ldots \times \mu_{r_p(m_n)},
\end{align*}
where $p$ is the characteristic of the field $k$. 
\end{proposition}

\begin{proof}
This follows from the observation that $\mu_{m_i} = \mu_{r_p(m_i)}$ for all $1 \leq i \leq n$ in an algebraically closed field of characteristic $p$.
\end{proof}

Circulant rings are closely related to group algebras of finite abelian groups.

\begin{theorem}
\label{THM:group_alg-circulant_ring}
Consider the group $G = \Z_{m_1} \times \ldots \times \Z_{m_n}$ where the $m_i$ are positive integers.
Then the following map is an isomorphism of $k$-algebras:
\begin{align*}
\Phi_G \colon k[G] \to \CR_{\vm / k}, \ \sum_{g \in G}  f_g \cdot g \mapsto \sum_{g \in G} f_g \cdot \prod_{i=1}^n  X_i^{g_i},
\end{align*}
where $g := (g_1, \ldots, g_n) \in G$ and $\vm = (m_1, \ldots ,m_n)$.
\end{theorem}

\begin{proof}
See \cite{subroto2024wedderburndecompositioncommutativesemisimple}.
\end{proof}

Theorem \ref{THM:group_alg-circulant_ring} applies to all finite abelian groups, due to the following fundamental result:

\begin{theorem}[\textbf{Fundamental theorem of finite abelian groups {\cite{lang2004algebra}}}]
\label{Thm-FundThmGroup}
Let $G$ be a finite abelian group.
Then there exists a unique tuple of integers $\vm := (m_1, \ldots ,m_n)$ such that $m_i \mid m_{i+1}$ for all $1 \leq i \leq n-1$, and 
\begin{align*}
G \cong \Z_{m_1} \times \ldots \times \Z_{m_n}.
\end{align*}
\end{theorem}

\begin{remark}
We call this unique tuple $\vm = (m_1, \ldots ,m_n)$ the \textbf{circulant parameters} of $G$.
\end{remark}

\subsection{Semisimple circulant rings}

We cover semisimple circulant rings and its simple components.

\begin{theorem}
\label{THM:Maschke}
A circulant ring $\CR_{\vm / k}$ is semisimple if and only if all entries of $\vm$ are coprime to the characteristic of $k$.
\end{theorem}

\begin{proof}
This is a direct consequence of Maschke's Theorem \cite{james2001representations}.
\end{proof}

A consequence of this theorem is that all circulant rings over fields of characteristic $0$ are semisimple.

The Krull-Remak-Schmidt decomposition of semisimple circulant rings is also known as the Wedderburn decomposition, and they have been well-studied in \cite{subroto2024wedderburndecompositioncommutativesemisimple}.
In this section, we briefly cover the main results of that paper, which will play an important part in constructing the Krull-Remak-Schmidt decomposition for general circulant rings.

\begin{definition}
\label{DEF:geom_group_action}
Define $k_{\vm} := k(\mu_{m_1}, \ldots, \mu_{m_n})$.
For any $\sigma \in \Gal(k_{\vm} / k)$ and for any $\bx := (x_1, \ldots, x_n) \in \V(\ad_{\vm})$, we define $\sigma(\bx) := (\sigma(x_1), \ldots, \sigma(x_n))$.
This induces the group action
\begin{align*}
\alpha_{\vm / k} \colon \Gal(k_{\vm} / k) \times \V(\ad_{\vm}), \ (\sigma, \bx) \mapsto \sigma(\bx),
\end{align*}
which we call the \textbf{geometric group action}.

For $\bx \in \V(\ad_{\vm})$, we denote $\Orb(\bx)$ as the orbit of $\bx$ under $\alpha_{\vm / k}$.
\end{definition}

\begin{remark}
Let $p$ be the characteristic of $k$.
Observe that $\alpha_{\vm / k}$ and $\alpha_{r_p(\vm) / k}$ are exactly the same maps since $\mu_{m} = \mu_{r_p(m)}$ in a characteristic $p$ field for all integers $m$.
\end{remark}

\begin{theorem}[\textbf{Semisimple ring decomposition}]
\label{THM:semisimple_ring_decomposition}
Let $\CR_{\vm / k}$ be a semisimple circulant ring, and let $S$ be a set of representatives of the orbits of $\alpha_{\vm / k}$.
Then we have the ring isomorphism
\begin{align*}
\CR_{\vm / k} \to \bigoplus_{\bx \in S} k(\bx), \ f \mapsto (f(\bx))_{\bx \in S}.
\end{align*} 
\end{theorem}

\begin{proof}
See \cite{subroto2024wedderburndecompositioncommutativesemisimple}.
\end{proof}

The simple submodule of $\CR_{\vm / k}$ corresponding to $k(\bx)$, in view of Theorem \ref{THM:important_result}, is exactly the submodule 
\begin{align*}
\{ f \in \CR_{\vm / k} \mid f(\by) = 0 \text{ for all } \by \in S \setminus \{ \bx \} \}.
\end{align*}
Let us denote this submodule by $\CR_{\vm / k}(\bx)$.

\begin{theorem}[\textbf{Wedderburn decomposition}]
\label{THM:Wedderburn_Dec}
Let $\CR_{\vm / k}$ be a semisimple circulant ring.
Let $S$ be a set of representatives of the orbits of $\alpha_{\vm / k}$.
Then $\CR_{\vm / k}$ admits the Wedderburn decomposition $\bigoplus_{\bx \in S}  \CR_{\vm / k}(\bx)$, where $\CR_{\vm / k}(\bx)$ is a simple $\CR_{\vm / k}$-submodule as defined above, which is equivalent to the field extension $k(\bx)$ of $k$.
\end{theorem}

\begin{proof}
This is immediate from Theorems \ref{THM:semisimple_ring_decomposition} and \ref{THM:important_result}.
\end{proof}

\section{Krull-Remak-Schmidt decomposition for circulant rings}
\label{SECTION:Krull-Remak-Schmidt}

Let $k$ be a field of characteristic $p$.
For $p = 0$, all circulant rings over $k$ are semisimple, for which Theorem \ref{THM:Wedderburn_Dec} provides a full description of their simple components.
When $p$ is prime, Theorem \ref{THM:Wedderburn_Dec} does not necessarily apply, as $\CR_{\vm / k}$ is not always semisimple (see Theorem \ref{THM:Maschke}).
The main goal of this section is to present the Krull-Remak-Schmidt decomposition of circulant rings in general, where we do not need to assume semisimpleness.
Unless stated otherwise, we assume that $p$ is a prime number.

\subsection{Some modular representation theory}

Modular representation theory studies group algebras which do not satisfy the conditions of Maschke's theorem.
In other words, it studies group algebras where the order of the group is not coprime to the characteristic of the base field.
In this subsection, we discuss some results of modular representation theory which are useful for this paper.
These are based on the lecture notes by Peter Schneider \cite{schneider2012modular}.

\begin{definition}
Let $R$ be a ring.
The \textbf{Jacobson radical} $\Jac(R)$ of $R$ consists of all elements $r \in R$ such that $r \cdot M = 0$ for all simple $R$-modules $M$. 
\end{definition}

\begin{theorem}
\label{THM:Jac_criteria}
Let $R$ be a commutative ring.
We have the following equivalent criteria for $\Jac(R)$:
\begin{itemize}
\item[1.]	$\Jac(R)$ is the smallest submodule of $R$ such that $R / \Jac(R)$ is a semisimple $R$-module;
\item[2.]	$\Jac(R)$ is the intersection of all maximal ideals of $R$;
\item[3.]	$\Jac(R)$ is the largest nilpotent ideal of $R$ (an ideal $I$ is nilpotent if $I^n = 0$ for some integer $n$).
\end{itemize}
\end{theorem}

\begin{proof}
See \cite[Theorem~6.1 on p.~658]{lang2004algebra}.
\end{proof}

The Jacobson radical is closely related to the Krull-Remak-Schmidt decomposition of an $R$-module $M$.
This is explained using projective modules.

\begin{definition}
An $R$-module $P$ is called a \textbf{projective module} if it is a summand of a free $R$-module.
\end{definition}

\begin{definition}
An $R$-module homomorphism $f : M \to M'$ is called \textbf{essentially surjective} if $f$ is surjective, but its restriction to any proper submodule of $M$ is not.
A \textbf{projective cover} of an $R$-module $M$ is a projective $R$-module $P$ together with an essential homomorphism $f : P \to M$.
\end{definition}

\begin{theorem}
We have the following statements regarding projective covers:
\begin{itemize}
\item	Every $R$-module $M$ has a projective cover, and it is unique up to isomorphism;
\item	If $P$ is a projective cover, then the quotient map $P \to P / \Jac(R)P$ is essential;
\item	The projective indecomposable modules are in bijection with the simple modules. 
In this bijection, the projective indecomposable $P$ corresponds with the simple module $M = P / \Jac(R)P$. 
The projective cover of $M$ is $P$.
\end{itemize}
\end{theorem}

\begin{proof}
The first statement is equivalent to Lemma 1.6.8. and Theorem 1.6.10. of \cite{schneider2012modular}.
The second and third statements are an immediate consequence of \cite[Proposition~1.7.4.]{schneider2012modular}
\end{proof}

The simple components of the semisimple module $R / \Jac(R)$ provide a lot of information about the indecomposable components of $R$.
We use this observation to construct the indecomposable components of general circulant rings, as we already have a good understanding of semisimple circulant rings.

\subsection{Non-semisimple circulant rings}

We discuss some properties of circulant rings in general where the entries of $\vm$ are positive integers which are not necessarily coprime to $p$, given that $p$ is a prime number.

\begin{lemma}
Every circulant ring is Noetherian and Artinian.
\end{lemma}

\begin{proof}
This is immediate from Theorems \ref{THM:groupalg_Noether} and \ref{THM:group_alg-circulant_ring}.
\end{proof}

\begin{lemma}
\label{LEM:3d_ring_iso}
We have $\ad_{\vm} \subseteq \ad_{r_p(\vm)} \subset A^n_k$ and $\CR_{\vm / k} / \ad_{r_p(\vm)} \cong \CR_{r_p(\vm) / k}$.
\end{lemma}

\begin{proof}
The identity $\ad_{\vm} \subseteq \ad_{r_p(\vm)} \subset A^n_k$ is immediate. 
By the third isomorphism theorem for rings, we get 
\begin{align*}
\CR_{r_p(\vm) / k} := A^n_k / \ad_{r_p(\vm)} \cong (A^n_k / \ad_{\vm}) / (\ad_{r_p(\vm)} / \ad_{\vm}) := \CR_{\vm / k} / \ad_{r_p(\vm)}.
\end{align*}
\end{proof}

\begin{lemma}
We have the identity $\Jac(\CR_{\vm}) = \ad_{r_p(\vm)}$.
\end{lemma}

\begin{proof}
From Lemma \ref{LEM:3d_ring_iso} and Theorem \ref{THM:Wedderburn_Dec}, $A_{\vm} / \ad_{r_p(\vm)}$ is semisimple, hence $\Jac(\CR_{\vm / k}) \subseteq \ad_{r_p(\vm)}$ by the first criterion of Theorem \ref{THM:Jac_criteria}.

On the other hand, observe that all $f \in \ad_{r_p(\vm)}$ are nilpotent elements in $\CR_{\vm / k}$. 
To see this, note that $f = \sum_{i = 1}^n  f_i(X) \cdot (X_i^{m_i} - 1)$, for some polynomials $f_i \in A^n_k$.
One can find a large enough integer $t$ such that by Newton's binomial theorem, $f^t$ consists of components each divisible by $(X_j^{m_j} - 1)^{p^{v_p(m_j)}}$ for some $1 \leq j \leq n$.
Thus $f^t$ vanishes modulo $\ad_{\vm}$, which means that $f$ is indeed nilpotent in $\CR_{\vm / k}$.
This means that $\ad_{r_p(\vm)} \subseteq \Jac(\CR_{\vm / k})$ by the third criterion of Theorem \ref{THM:Jac_criteria}.
Hence $\Jac(\CR_{\vm / k}) = \ad_{r_p(\vm)}$.
\end{proof}

An interesting case is when all the entries of $m$ are powers of $p$.

\begin{lemma}
\label{LEMMA:local_condition}
Let $m$ be an $n$-tuple whose entries are a power of $p$.
Then $\CR_{\vm / k}$ is a local ring, with maximal ideal $\m := (X_1 - 1, X_2 - 1, \ldots, X_n - 1)$ and residue field $k$.
\end{lemma}

\begin{proof}
When $m$ is an $n$-tuple whose entries are powers of $p$, then $\CR_{\vm / k}$ is isomorphic to a group algebra over a commutative $p$-group, which is local by \cite[Proposition~2.9.7.]{schneider2012modular}.
\end{proof}

\begin{remark}
The converse statement is also true, which we prove in Corollary \ref{COR:non-locality}.
\end{remark}

\subsection{Field embedding}

Let 
\begin{align*}
Q : \CR_{\vm / k} \to \CR_{\vm / k} / \Jac(\CR_{\vm / k}) \cong \CR_{r_p(\vm) / k},
\end{align*}
be the quotient map.
From Theorem \ref{THM:semisimple_ring_decomposition}, we know that $\CR_{r_p(\vm) / k}$ is isomorphic to the direct product of fields $\bigoplus_{\bx \in S} k(\bx)$, where $S$ is a set of representatives of the group action $\alpha_{\vm / k}$ in Defintion \ref{DEF:geom_group_action}.

The goal of this subsection is to construct injective group homomorphisms $k(\bx) \to \CR_{\vm / k}$ for all $\bx \in S$ which preserve multiplication, and are also $k$-linear maps.
Observe that by Theorem \ref{THM:Wedderburn_Dec}, the field $k(\bx)$ corresponds to the set
\begin{align*}
\{ f \in \mathcal{R}_{r_p(\vm)} \mid f(\by) = 0 \text{ for all } \by \in S \setminus \{ \bx \} \}.
\end{align*}
In words, the elements in $k(\bx)$ can be identified by the polynomials in $\CR_{\vm / k}$ with partial degrees strictly smaller than $r_p(\vm)$, such that it vanishes over all representatives in $S$ other than $\bx$.
We use this identification of elements of $k(\bx)$ to construct such group homomorphisms.

\begin{lemma}
\label{LEMMA:power_p_vanish}
Let $t \geq \sum_{i=1}^n v_p(m_i)$.
Then for any $f(X_1, \ldots, X_n) \in \ad_{r_p(\vm)}$, we have $f(X_1^{p^t}, \ldots, X_n^{p^t}) = 0$ in $\CR_{\vm / k}$.
\end{lemma}

\begin{proof}
Since $f \in \ad_{r_p(\vm)}$, $f$ is of the form $f = \sum_{i=1}^n f_i(X_1, \ldots, X_n) \cdot \left( X_i^{r_p(m_i)} - 1 \right)$ where $f_i \in \CR_{\vm / k}$.
Observe that
\begin{align*}
f(X_1^{p^t}, \ldots, X_n^{p^t}) 
&= \sum_{i=1}^n f_i(X_1^{p^t}, \ldots, X_n^{p^t}) \cdot \left( X_i^{p^t \cdot r_p(m_i)} - 1 \right) \\
&= \sum_{i=1}^n f_i(X_1^{p^t}, \ldots, X_n^{p^t}) \cdot \left( X_i^{r_p(m_i) \cdot p^{v_p(m_i)}} - 1 \right)^{p^{t - v_p(m_i)}} \\
&= 0,
\end{align*}
since the $X_i^{r_p(m_i) \cdot p^{v_p(m_i)}} - 1 = X_i^{m_i} - 1$ vanish in the ideal $\ad_{\vm}$.
\end{proof}

\begin{lemma}
\label{LEM:t_conditions}
There exists an integer $t > 0$ such that $p^t \equiv 1 \bmod r_p(m_i)$ for all $1 \leq i \leq n$, and such that $t \geq \sum_{i=1}^n v_p(m_i)$.
\end{lemma}

\begin{proof}
Let us fix some $1 \leq i \leq n$.
Since $p$ is coprime to $r_p(m_i)$, there exists a positive integer $t_i$ such that $p^{t_i} \equiv 1 \bmod r_p(m_i)$.
Observe that every this is true for every multiple $t_i$.
As a result, we have that $p^{\prod_{i=1}^n t_i} \equiv 1 \bmod r_p(m_i)$ for all $1 \leq i \leq n$.
Now simply find an integer $c > 0$ such that $c \cdot \prod_{i=1}^n t_i > \sum_{i=1}^n v_p(m_i)$, and set $t := c \cdot \prod_{i=1}^n t_i$.
Then $t$ satisfies all the conditions of the lemma.
\end{proof}

\begin{proposition}
\label{PROP:field_embed}
Let $t$ be some fixed integer satisfying the conditions in Lemma \ref{LEM:t_conditions}.
Then the map 
\begin{align*}
\iota_{\bx} \colon k(\bx) \to \CR_{\vm / k}, \ f(X_1, \ldots, X_n) \mapsto f \left( X_1^{p^t}, \ldots, X_n^{p^t} \right),
\end{align*}
is an injective $k$-linear map which preserves multiplication.
Moreover, $Q \circ \iota_{\bx} : k(\bx) \to k(\bx)$ is the identity map.
\end{proposition}

\begin{proof}
To ease notation in this proof, we denote $(X_1^c, \ldots, X_n^c)$ by $X^c$ for any positive integer $c$.
Observe that $k$-linearity of $\iota_{\bx}$ is immediate, since we simply replace $X$ by $X^{p^t}$.

Let $h \in k(\bx)$ such that $h$ equals $f \cdot g$ within the field $k(\bx)$.
This means that in $\CR_{\vm / k}(\bx)$, there exists a polynomial $h^* \in \ad_{r_p(\vm)}$ such that $h = f \cdot g +  h^*$.
Now we have
\begin{align*}
\iota_{\bx}(f \cdot g) 
= \iota_{\bx}(h) 
= h(X^{p^t}) 
= (f \cdot g)(X^{p^t}) +  h^*(X^{p^t})
= f(X^{p^t}) \cdot g(X^{p^t}) + 0 
= \iota_{\bx}(f) \cdot \iota_{\bx}(g),
\end{align*}
here the fourth equation is due to Lemma \ref{LEMMA:power_p_vanish}, and also due to the fact that for polymomials the equation $(f \cdot g)(X^c) = f(X^c) \cdot g(X^c)$ holds up for any positive integer $c$.
Hence $\iota_{\bx}$ preserves multiplication.

Observe that for any $\by \in \V(\ad_{\vm})$, we have that $y_i^{p^t} = y_i$ since $p^t \equiv 1 \bmod m_i$ for all $1 \leq i \leq n$.
As a result, we have for every $f \in k(\bx)$ and for every $\by \in S$ that $\iota_{\bx}(f)(\by) = f(\by)$.
This implies that $Q(\iota_{\bx}(f)) = f$, which shows that $Q \circ \iota_{\bx}$ is the identity, which in particular implies that $\iota_{\bx}$ is injective.
\end{proof}

\begin{lemma}
\label{LEMMA:field_vanish}
Let $\by \neq \bx$ be representatives of distinct orbits of the geometric group action $\alpha_{\vm / k}$.
Then for all $f_{\bx} \in \im(\iota_{\bx})$ and $f_{\by} \in \im(\iota_{\by})$, we have the identity $f_{\bx} \cdot f_{\by} = 0$.
\end{lemma}

\begin{proof}
Let $f'_{\bx} \in k(\bx)$ and $f'_{\by} \in k(\by)$ such that $\iota_{\bx}(f'_{\bx}) = f_{\bx}$ and $\iota_{\by}(f'_{\by}) = f_{\by}$. 
Observe that the polynomial $f'_{\bx} \cdot f'_{\by}$ vanishes everywhere on $\V(\ad_{\vm})$, hence it must be contained in $\ad_{r_p(\vm)}$.

Note that we have the identities
\begin{align*}
f_{\bx} \cdot f_{\by} 
= \iota_{\bx}(f'_{\bx}) \cdot \iota_{\by}(f'_{\bx}) 
:= f'_{\bx}(X_1^{p^t}, \ldots, X_n^{p^t}) \cdot f'_{\by}(X_1^{p^t}, \ldots, X_n^{p^t}) 
= (f'_{\bx} \cdot f'_{\by})(X_1^{p^t}, \ldots, X_n^{p^t}),
\end{align*}
which is contained in $\ad_{\vm}$ by Lemma \ref{LEMMA:power_p_vanish}.
\end{proof}

\subsection{Indecomposable components}

In this subsection, we construct local rings which are the components of the local ring decomposition of $\CR_{\vm / k}$.

\begin{definition}
\label{DEF:importan_local_ring}
Let $S \subseteq \V(\ad_{\vm})$ be a set of representatives of the orbits of $\alpha_{\vm / k}$ as introduced in Definition \ref{DEF:geom_group_action}.
For $\bx \in S$, we define the corresponding local coordinate ring
\begin{align*}
\CR_{\vm / k} [\bx] := k(\bx)[Y_1, \ldots, Y_n] / \left(Y_1^{p^{v_p(m_1)}}, \ldots, Y_n^{p^{v_p(m_n)}} \right),
\end{align*}
where $p$ is the characteristic of $k$.
If $ p = 0$, then we define $\CR_{\vm / k} [\bx] := k(\bx)$.
\end{definition}

\begin{theorem}
\label{THEOREM:iso_construct}
The map
\begin{align*}
F_{\bx} \colon \CR_{\vm / k}[\bx] \to \CR_{\vm / k}, \ \sum_{j : j \ll p^{v_p(\vm)}}   f_j \cdot Y^j \mapsto \sum_{j : j \ll p^{v_p(\vm)}}  \iota_{\bx}(f_j) \cdot \Y_{r_p(\vm)}^j,
\end{align*}
is a well-defined injective $k$-linear map which preserves multiplication.
Here, $\Y_{r_p(\vm)}^j$ refers to the notation in Definition \ref{DEF:polyrep}.
\end{theorem}

\begin{proof}
We divide the proof in two parts.

\paragraph{\textbf{Well-definedness}}
Consider the map
\begin{align*}
\wt{F_{\bx}} \colon k(\bx)[Y_1, \ldots, Y_n] \to \CR_{\vm / k}, \ \sum_{j \in \Z_{\geq 0}^n} f_j \cdot Y^j \mapsto \sum_{j  \in \Z_{\geq 0}^n}  \iota_{\bx}(f_j) \cdot \Y_{r_p(\vm)}^j.
\end{align*}
This map is a $k$-linear map, and it preserves multiplication since $\iota_{\bx}$ also satisfies these conditions (Proposition \ref{PROP:field_embed}).
The kernel of $\wt{F_{\bx}}$ contains the ideal $\left(  Y_1^{p^{v_p(m_1)}}, \ldots, Y_n^{p^{v_p(m_n)}} \right)$, since 
\begin{align*}
\wt{F_{\bx}} \left( Y_i^{p^{v_p(m_i)}} \right) 
= \left( X_i^{r_p(m_i)} - 1 \right)^{p^{v_p(m_i)}} 
= X_i^{r_p(m_i) \cdot p^{v_p (m_i)}} - 1 
= X_i^{m_i} - 1,
\end{align*}
are exactly the generators of $\ad_{\vm}$, which is the ideal qoutient of $\CR_{\vm / k}$.
Hence $\wt{F_{\bx}}$ induces the map $F_{\bx}$, which is thus also a $k$-linear map preserving multiplication.

\paragraph{\textbf{Injectivity}}
We show that $F_{\bx}$ is an injective map.

Assume to the contrary that this is not true, then there must exist $f := \sum_{j : j \ll p^{v_p(\vm)}}   f_j \cdot Y^j \in \mathcal{R}_{ p^{v_p(\vm)}}$ such that $F_{\bx}(f) \in \ad_{r_p(\vm)}$.
Take any term of $f$ which we denote as $f_j \cdot Y^j$ for some $n$-tuple $j$ corresponding to the term.
Note that there exists a non-zero $\ovl{f_j} \in \mathcal{R}_{r_p(\vm)}$ and $O \in \ad_{r_p(\vm)}$ such that $\iota_{\bx}(f_j)$ can be written as $\ovl{f_j} + O$.
This in particular means that 
$F_{\bx}(f_j \cdot Y^j) = (\ovl{f_j} + O) \cdot \Y_{r_p(\vm)}^j = \ovl{f_j} \cdot \Y_{r_p(\vm)}^j + O \cdot \Y_{r_p(\vm)}^j$ 
where $O \cdot \Y_{r_p(\vm)}^j \notin \mathcal{R}_{r_p(\vm) \cdot j}$.
By Theorem \ref{THM:polyStructureTheorem}, there exists a unique set of $n$-tuples $\mathcal{N}$ together with a unique family of corresponding non-zero polynomials $\{ g_w : w \in \mathcal{N} \} \subset \mathcal{R}_{r_p(\vm)}$ such that 
\begin{align*}
O \cdot \Y_{r_p(\vm)}^j = \sum_{w \in \mathcal{N}} g_w \cdot \Y_{r_p(\vm)}^w \in \bigoplus_{w \in \mathcal{N}} \mathcal{R}_{r_p(\vm)} \Y_{r_p(\vm)}^w.
\end{align*}
Observe that since $O \in \ad_{r_p(\vm)}$, we must have that $w > j$ for all $w \in \mathcal{N}$.
Thus if we choose the term $f_j \cdot Y^j$ as the term of the lowest total degree of $f$ (there can be multiple of such terms), then we are sure that $F_{\bx}(f_j Y^j)$ has a non-zero $\Y^j_{r_p(\vm)}$-component, and that no other term of $f$ mapped under $F_{\bx}$ has a non-zero $\Y^j_{r_p(\vm)}$-component.
This implies however that $F_{\bx}(f)$ is non-zero by Theorem \ref{THM:polyStructureTheorem}, which is a contradiction, hence $F_{\bx}$ must be injective.
\end{proof}

\begin{remark}
The map $\iota_{\bx}$ is not a ring homomorphism, as the unit element of $k(\bx)$ is not mapped to the unit element of $\CR_{\vm / k}$.
\end{remark}

\begin{definition}
\label{DEF:R_m(x)}
We denote the image of $F_{\bx}$ by $\CR_{\vm / k}(\bx)$, which is a $k$-vector space.
\end{definition}

\begin{lemma}
\label{LEMMA:ZeroIntersect}
Let $\bx, \bx' \in \V(\ad_{\vm})$ be two elements which represent different orbits under $\alpha_{\vm / k}$.
Then $\CR_{\vm / k}(\bx) \cap \CR_{\vm / k}(\bx') = \{ 0 \}$.
Moreover, if $f \in \CR_{\vm / k}(\bx)$ and $g \in \CR_{\vm / k}(\bx')$, then $f \cdot g = 0$.
\end{lemma}

\begin{proof}
Assume to the contrary that there exist non-zero $f \in \CR_{\vm / k}[\bx]$ and $f' \in \CR_{\vm / k}[\bx']$ such that $F_{\bx}(f) = F_{\bx'}(f')$.
Then $\iota_{\bx}(f_j) = \iota_{\bx'}(f'_j)$ for all $n$-tuples $j \geq (0, \ldots, 0)$.
Since $f$ and $f'$ are non-zero, there exists an $n$-tuple $j^*$ such that $f_{j^*}$ is non-zero.
Observe that $\iota_{\bx}(f_{j^*})(\bx) \neq 0$ since $f_{j^*}$ is non-zero, but $\iota_{\bx'}(f'_{j^*})(\bx) = 0$ by assumption of $k(\bx')$.
Hence $\iota_{\bx}(f_j) \neq \iota_{\bx'}(f'_j)$, which is a contradiction.

The last statement is a direct consequence of Lemma \ref{LEMMA:field_vanish}.
\end{proof}

\begin{theorem}[\textbf{Local ring decomposition}]
\label{THM:direct_sum}
Let $k$ be a field of characteristic $p$, and consider the product ring $\bigoplus_{\bx \in S} \CR_{\vm / k}[\bx]$.
Then the map
\begin{align*}
F \colon \bigoplus_{\bx \in S}  \CR_{\vm / k}[\bx] \to \CR_{\vm / k}, \ (f_{\bx})_{\bx \in S} \mapsto \sum_{\bx \in S} F_{\bx}(f_{\bx}),
\end{align*}
is a ring isomorphism.
\end{theorem}

\begin{proof}
We split the proof in two parts.

\paragraph{\textbf{Bijectivity}}
As a result of Lemma \ref{LEMMA:ZeroIntersect}, the image of $F$ is the direct sum $\bigoplus_{\bx \in S} \CR_{\vm / k}(\bx)$ within $\CR_{\vm / k}$.
This implies that $F$ is injective, as all the maps $F_{\bx}$ are injective.

Let us show surjectivity.
Since $F$ is a $k$-linear map, we conclude that $\bigoplus_{\bx \in S} \CR_{\vm / k}(\bx)$ is a $k$-subspace of $\CR_{\vm / k}$ viewed as a $k$-vector space.
Observe that 
\begin{align*}
\dim_k \left( \bigoplus_{\bx \in S}  \CR_{\vm / k}(\bx) \right)
&= \sum_{\bx \in S} \dim_k(\CR_{\vm / k}[\bx]) 
= \sum_{\bx \in S} \left( [k(\bx) : k] \cdot  \prod_{i = 1}^n p^{v_p(m_i)} \right)\\
&= \prod_{i = 1}^n p^{v_p(m_i)} \cdot \left( \sum_{\bx \in S} [k(\bx) : k] \right).
\end{align*}
Note that $\sum_{\bx \in S} [k(\bx) : k] = \sum_{\bx \in S} \# \Orb(\bx) = \# \V(\ad_{r_p(\vm)}) = \prod_{i=1}^n r_p(m_i)$, hence
\begin{align*}
\dim_k \left( \bigoplus_{\bx \in S}  \CR_{\vm / k}[\bx] \right)
= \left( \prod_{i = 1}^n p^{v_p(m_i)} \right) \cdot \left( \prod_{j=1}^n r_p(m_j) \right)
= \prod_{i = 1}^n r_p(m_i) \cdot p^{v_p(m_i)} = \prod_{i=1}^n m_i.
\end{align*}
On the other hand, $\dim_k(\CR_{\vm / k}) = \prod_{i = 1}^n m_i$, hence $\dim_k \left( \bigoplus_{\bx \in S}  \CR_{\vm / k}[\bx] \right) = \dim_k(\CR_{\vm / k})$.
This implies that $F$ is also surjective, hence a bijection.

\paragraph{\textbf{Ring homomorphism}}
Here we show that $F$ is a ring homomorphism.

Observe that $F$ preserves addition, as this is the sum of $F_{\bx}$ which are all (additive) group homomorphisms.

Let $f = (f_{\bx})_{\bx \in S}, g = (g_{\bx})_{\bx \in S} \in \bigoplus_{\bx \in S} \CR_{\vm / k}(\bx)$, then
\begin{align*}
F(f) \cdot F(g) 
&= \left( \sum_{\bx \in S} F_{\bx}(f_{\bx}) \right)  \left( \sum_{\bx \in S} F_{\bx}(g_{\bx}) \right)
= \sum_{\bx \in S} F_{\bx}(f_{\bx}) \cdot F_{\bx}(g_{\bx}) 
= \sum_{\bx \in S} F_{\bx}(f_{\bx} \cdot g_{\bx})    \\
&= F(f \cdot g),
\end{align*}
where the second equation is due to Lemma \ref{LEMMA:ZeroIntersect}, and the third equation due to the fact that the $F_{\bx}$ preserve multiplication.
Hence $F$ preserves multiplication.

Let $1_S$ be the identity in $\bigoplus_{\bx \in S}  \CR_{\vm / k}[\bx]$.
We want to show that $F(1_S)$ is the identity in $\CR_{\vm / k}$.
Let $f$ be any element in $\CR_{\vm / k}$.
Since $F$ is bijective, there exists a unique element $f^* \in \bigoplus_{\bx \in S}  \CR_{\vm / k}[\bx]$ such that $F(f^*) = f$.
Note that $F(1_S) \cdot f = F(1_S) F(f^*) = F(1_S f^*) = F(f^*) = f$.
Since this is true for all $f \in \CR_{\vm / k}$, $F(1_S)$ must be the identity in $\CR_{\vm / k}$.

In conclusion, $F$ is a well-defined bijective ring homomorphism, which makes it a ring isomorphism.
\end{proof}

\begin{corollary}
\label{COR:non-locality}
Let $k$ be a field of characteristic $p$, where $p$ is either $0$ or a prime number. 
If there exist entries of the $n$-tuple $\vm$ which are not powers of $p$, then $\CR_{\vm / k}$ is not local.
\end{corollary}

\begin{proof}
Since not all entries of $\vm$ are coprime, the set of orbit representatives $S$ of $\alpha_{\vm / k}$ consists of more than one element.
Hence by Theorem \ref{THM:direct_sum}, $\CR_{\vm / k}$ is a direct sum of multiple local rings, which cannot be local.
\end{proof}

\begin{theorem}[\textbf{Krull-Remak-Schmidt for circulant rings}]
\label{THM:KRM_circulant_rings}
Let $\CR_{\vm / k}$ be a circulant ring over a field $k$ of characterstic $p$.
Let $S$ be a set of representatives of the orbits of $\alpha_{\vm / k}$ introduced in Definition \ref{DEF:geom_group_action}.
Then $\CR_{\vm / k}$ admits the module decomposition $\bigoplus_{\bx \in S}  \CR_{\vm / k}(\bx)$, where $\CR_{\vm / k}(\bx)$ is introduced in Definition \ref{DEF:R_m(x)}.
Moreover, for every $\bx \in S$, $\CR_{\vm / k}(\bx)$ is an indecomposable $\CR_{\vm / k}$-module, which for $p$ prime is equivalent to the local circulant ring
\begin{align}
\label{EQ:KRS_circ}
k(\bx)[Y_1,...,Y_n] / \left(  Y_1^{p^{v_p(m_1)}} - 1, \ldots, Y_n^{p^{v_p(m_n)}} - 1 \right).
\end{align}
For $p = 0$, $\CR_{\vm / k}(\bx)$ is isomorphic to $k(\bx)$.
\end{theorem}

\begin{proof}
All the claims except for Eq. (\ref{EQ:KRS_circ}) are direct consequences of Theorem \ref{THM:direct_sum} and Corollary \ref{COR:important_result}.
We only need to show that $\CR_{\vm / k}[\bx]$ as defined in \ref{DEF:importan_local_ring} is isomorphic to Eq. (\ref{EQ:KRS_circ}).
When $p$ is prime, this is simply a matter of applying the variable transformation $Y_i \mapsto Y_i - 1$.
The case for $p = 0$ is immediate.
\end{proof}

\section{Group algebras over finite fields}
\label{SECTION:Application}

In this section, we combine the above results together with the results in \cite{subroto2024wedderburndecompositioncommutativesemisimple} to provide a full overview of the Krull-Remak-Schmidt decomposition of circulant rings over finite fields.
We denote a finite field of order $q$ by $\F_q$.

We also give an explicit example using $\CR_{(5,5,2^l) / \F_2}$, where $l \in \Z_{>0}$.
This particular circulant ring is interesting, as it is used in the construction of the $\theta$-function in the cryptographic primitive {\sc Keccak}-$f$ \cite{DBLP:journals/iacr/BertoniDPA15}.

\subsection{Krull-Remak-Schmidt over finite fields}

We introduce some additional notation related to the ring of integers modulo $m$.
The multiplicative group of $\Z_m$ is denoted as $\Z_m^*$, 
We denote the order of $\Z_m^*$ by $\varphi(m)$, which is known as the Euler's totient function.
Observe that $a \in \Z_m$ is contained in $\Z_m^*$ if and only if $a$ is coprime to $m$.
For $a \in \Z_m^*$, we denote the algebraic order of $a$ as $\ord_m(a)$.

We define $\Div_m$ as the set of all divisors of $m$ including $1$ and $m$ itself.
For an $n$-tuple $\vm \in \Z_{>0}^n$, we define $\Div_{\vm} := \Div_{m_1} \times \ldots \times \Div_{m_n} \subset \Z_{>0}^n$.

\begin{theorem}[\textbf{Circulant Krull-Remak-Schmidt decomposition over finite fields}]
\label{THM:KRM_finite_fields}
Let $\F_q$ be a finite field of characteristic $p$, and consider the circulant ring $\CR_{\vm / \F_q}$.
For any $\vd := (d_1, \ldots, d_n) \in \Div_{r_p(\vm)}$, define the expressions $\nu_{\vm / \F_q}(\vd) := \lcm_{1 \leq i \leq n} (\ord_{r_p(m_i) / d_i}(q))$ and $\eta_{\vm / \F_q}(\vd) := \frac{\prod_{i=1}^n \varphi(r_p(m_i))}{\nu_{\vm,q}(\vd)}$.
Then we have the decomposition:
\begin{align*}
\CR_{\vm / \F_q} 
\cong
\bigoplus_{\vd \in \Div_{r_p(\vm)}} \left( \F_{q^{\nu_{\vm / \F_q}(\vd)}} [Y_1, \ldots, Y_n] / \left( Y_1^{p^{v_p(m_1)}} - 1, \ldots,  Y_n^{p^{v_p(m_n)}} - 1  \right) \right)^{\eta_{\vm / \F_q}(\vd)}.
\end{align*}
\end{theorem}

\begin{proof}
This is immediate from Theorem \ref{THM:KRM_circulant_rings}, and the Wedderburn decomposition of circulant rings over finite fields presented in \cite{subroto2024wedderburndecompositioncommutativesemisimple}.
\end{proof}

\subsection{An example}

Consider the case 
\begin{align*}
\CR_{(5,5,2^l) / \F_2} := \F_2[X_1, X_2, X_3] / (X_1^5 - 1, X_2^5 - 1, X_3^{2^l} - 1),
\end{align*}
for some positive integer $l$.
Observe that $r_2((5,5,2^l)) = (5,5,1)$, hence 
\begin{align*}
\Div_{r_2((5,5,2^l))} := \{ (1,1,1), (1,5,1), (5,1,1), (5,5,1) \}.
\end{align*}
In view of Theorem \ref{THM:KRM_finite_fields}, we compute $\prod_{i=1}^3 \varphi(d_i)$, $\nu_{(5,5,2^l) / \F_2}(\vd)$ and $\eta_{(5,5,2^l) / \F_2}(\vd)$ for each $\vd \in \Div_{r_2((5,5,2^l))}$:

\begin{table}[!h]
\begin{center}
	\begin{tabular}{| l | l | l | l |}
	\hline
	$\vd = (d_1,d_2,d_3)$	& $\prod_{i=1}^3 \varphi(d_i)$	& $\nu_{(5,5,2^l) / \F_2}(\vd)$	& $\eta_{(5,5,2^l) / \F_2}(\vd)$ 	\\ \hline
	$(1,1,1)$				& $1$							& $1$ 							& $1$						\\
	$(1,5,1)$				& $4$							& $4$							& $1$						\\
	$(5,1,1)$				& $4$							& $4$							& $1$						\\
	$(5,5,1)$				& $16$							& $4$							& $4$						\\
	\hline
	\end{tabular}
\caption{Values of $\nu_{(5,5,1) / \F_2}(\vd)$ and $\eta_{(5,5,1) / \F_2}(\vd)$}
\end{center}
\end{table}

By Theorem \ref{THM:KRM_finite_fields}, the number of indecomposable components of $\CR_{(5,5,2^l) / \F_2}$ equals:
\begin{align*}
\sum_{d \in \Div_{r_p(\vm)}} \eta_{(5,5,1) / \F_2}(\vd) = 1 + 1 + 1 + 4  = 7.
\end{align*}
Again by Theorem \ref{THM:KRM_finite_fields}, the indecomposable components of $\CR_{(5,5,2^l) / \F_2}$ consist up to isomorphism of the following rings:
\begin{align*}
\F_2[Y_1, Y_2, Y_3] / (Y_1 - 1, Y_2 - 1, Y_3^{2^l} - 1) &\cong \F_2[Y] / (Y^{2^l} - 1), \\
\F_{2^4}[Y_1, Y_2, Y_3] / (Y_1 - 1, Y_2 - 1, Y_3^{2^l} - 1) &\cong \F_{2^4}[Y] / (Y^{2^l} - 1),
\end{align*}
hence we have the decomposition:
\begin{align*}
\CR_{(5,5,2^l) / \F_2} \cong \F_2[Y] / (Y^{2^l} - 1) \oplus \left(  \F_{2^4}[Y] / (Y^{2^l} - 1) \right)^6.
\end{align*}

\begin{acknowledgements}
I would like to thank Jan Schoone for providing this manuscript with useful feedback.

I would like to thank my PhD-supervisor prof. dr. Joan Daemen for providing me with research topics leading to this paper.

This work was supported by the European Research Council under the ERC advanced grant agreement under grant ERC-2017-ADG Nr.\ 788980 ESCADA. 
\end{acknowledgements}

\bibliographystyle{amsalpha} 
\bibliography{reference_paper5}

%\begin{thebibliography}{PTW02} % '2nd argument contains the widest acronym'
%\bibitem[Lam94]{Lamport}
%  L.~Lamport, \emph{\LaTeX: A document preparation system \textup{(}%
%  updated for \LaTeXe\textup{)}} (Addison--Wesley, New York, 1994).
%\bibitem[PTW02]{PRL}
%  T.~Prokopec, O.~T\"ornkvist and R.~P.~Woodard, \emph{Photon mass
%  from inflation}, Phys.\ Rev.\ Lett.\ \textbf{89} (2002), 101301.
%\end{thebibliography}

\end{document}